\documentclass[brochure,english,12pt]{smfbourbaki}

\usepackage[T1]{fontenc}
\usepackage{lmodern,amssymb,bm}

\usepackage[english]{babel}

\usepackage[utf8]{inputenc}

\usepackage[
backend=bibtex,
style=authoryear, 
citestyle=authoryear-comp,
maxnames=5
]{biblatex}
\usepackage{csquotes}

\DeclareDelimFormat{nameyeardelim}{\addcomma\space}

\DeclareNameAlias{sortname}{given-family}

\renewcommand{\bibnamedash}{\leavevmode\raise3pt\hbox to3em{\hrulefill}\space}

\AtEveryBibitem{%
  \clearfield{issn} 
  \clearfield{doi} 

  \ifentrytype{online}{}{
    \clearfield{url}
  }
}

\renewbibmacro{in:}{%
    \ifentrytype{article}{}{\printtext{\bibstring{in}\intitlepunct}}}

\DeclareFieldFormat[article,periodical,inreference]{number}{\mkbibparens{#1}}
\DeclareFieldFormat[article,periodical,inreference]{volume}{\mkbibbold{#1}}
\renewbibmacro*{volume+number+eid}{%
    \printfield{volume}%
    \setunit*{\addthinspace}
    \printfield{number}%
    \setunit{\addcomma\space}%
    \printfield{eid}}

\DeclareFieldFormat[article,inbook,incollection]{title}{\enquote{#1}\addcomma} 

\date{Mai 2021}
\bbkannee{73\textsuperscript{e} ann\'ee, 2020--2021}  
\bbknumero{1179}                                      

\title{Asymptotic counting of minimal surfaces in hyperbolic $3$-manifolds} 
\subtitle{according to Calegari, Marques and Neves}

\author{François Labourie}
\address{LJAD, Université Côte d'Azur}
\email{
  francois.labourie@univ-cotedazur.fr}

\usepackage{graphicx}
\usepackage{subcaption}
\usepackage{enumerate} \usepackage[colorlinks=true, pdfstartview=FitV,
linkcolor=blue, citecolor=red, urlcolor=blue]{hyperref}
\usepackage{amsmath,amsfonts,latexsym,amssymb}
\usepackage{braket}
\usepackage{microtype}
\usepackage{marginnote}
\usepackage{accents}

\newcommand{\T}{\mathsf T}
\renewcommand{\leq}{\leqslant}
\renewcommand{\geq}{\geqslant}
\renewcommand{\epsilon}{\varepsilon}

\newcommand{\sld}{\mathsf{SL}_2(\mathbb R)}
\newcommand{\psld}{\mathsf{PSL}_2(\mathbb R)}
\newcommand{\hh}{{\bf H}^2}

\newcommand{\coloneqq}{\mathrel{\mathop:}=}
\newcommand{\defeq}{\coloneqq}
\newcommand{\eqqcolon}{=\mathrel{\mathop:}}
\newcommand{\eqdef}{\eqqcolon}

\newcommand{\vol}{{\rm d}\operatorname{vol}}
\newcommand{\Vol}{\operatorname{Vol}}
\newcommand{\area}{{\rm d}\operatorname{area}}
\newcommand{\Area}{\operatorname{Area}}
\newcommand{\pslc}{\mathsf{PSL}(2,\mathbb C)}
\newcommand{\hhh}{\mathbf{H}^3}
\newcommand{\bhhh}{\partial_\infty\hhh}
\newcommand{\MinA}{\operatorname{MinArea}} 
\newcommand{\Env}{\operatorname{Env}} 
\newcommand{\iso}{\operatorname{Isom}} 
\newcommand{\qf}{quasi-Fuchsian}

\newtheorem{question}{Question}

\newcommand{\seq}[1]{ \{#1_m\}_{m\in\mathbb N}}
\newcommand{\seqm}[1]{ \{{#1}\}_{m\in\mathbb N}}

\begin{document}

\renewcommand{\labelenumi}{\theenumi}
\renewcommand{\theenumi}{(\roman{enumi})}%

\maketitle
\begin{flushleft}
\advance\leftskip .3\textwidth
	\it To Srishti Dhar Chatterji, \\
	who attended séminaire Bourbaki for most of half a century.
\end{flushleft}
\tableofcontents

\section*{Introduction}

The study of the geodesic flow in closed negatively curved manifolds is a beautiful mix of topology, Riemannian geometry, geometric group theory and ergodic theory. We know in this situation that closed geodesics are in one-to-one correspondence with conjugacy classes of elements of the fundamental group, or equivalently, with the set of homotopy classes of maps of circles in the manifold. Even though  closed geodesics are infinite in number, we have a good grasp ---thanks to the notion of {\em topological entropy}--- of how the number of these geodesics grows with respect to the length. We also have a computation of this topological  entropy in hyperbolic spaces by \textcite{Bowen:1972} and \textcite{Margulis:1969ve} and rigidity results for this entropy by \textcite{Besson:1995um} and \textcite{Hamenstadt:1990tx}.

While the statements of this first series of results  seem to deal only with  closed geodesics, the foliation of the unit tangent bundle by orbits of the geodesic flow plays a fundamental role. The study of invariant measures by the geodesic  flow is a crucial tool, and the equidistribution of closed geodesics by \textcite{Bowen:1972}  and  \textcite{Margulis:1969ve} for hyperbolic manifolds  a central result.  We refer to section~\ref{sec:par} for more precise definitions, results and references.  

For many reasons ---as we discuss in section~\ref{sec:tg}--- closed totally geodesic submanifolds of dimension at least $2$ are quite rare. However, in constant curvature, the foliation of the Grassmannian of $k$-planes coming from totally geodesic planes is a natural generalization of the geodesic flow and several crucial results of \textcite{Ratner:1991wl,Ratner:1991tu,Shah:1991wr} as well as   \textcite{McMullen:2017ub} describe closed invariant sets and invariant measures. This foliation stops to make sense in variable curvature, at least far away from the constant curvature situation, although for metrics close to hyperbolic ones, a result by \textcite{Gromov:1991uv} ---see also \cite{Lowe:2020vu}--- shows that the foliation of the Grassmann bundle persists when one replaces totally geodesic submanifolds by minimal ones.

If we move in the topological direction, going from circles to surfaces, Kahn--Markovi\'c Surface  Subgroup Theorem (\cite{Kahn:2009wh}) provides the existence of many surface subgroups in the fundamental group of a hyperbolizable $3$-manifold $M$. A subsequent result of \textcite{Kahn:2010uo}  gives an asymptotic of the number of these surface groups with respect to the genus ---see  Theorem~\ref{theo:KMcount}. 

However this  asymptotic counting does not involve the underlying Riemannian geometry as opposed to the topological entropy that we discussed in the first paragraph. The next step is to use  fundamental results of \textcite{Schoen:1979} and \textcite{Sacks:1982}, which tells us that every such surface group can be realized by a minimal surface ---although non necessarily uniquely.

In \textcite{Calegari:2020uo}, the authors propose a novel idea:  count asymptotically with respect to the area these minimal surfaces, but when the boundary at infinity of those minimal surfaces becomes more and more circular, or more precisely are  $K$-quasicircles, with $K$ approaching~$1$.  The precise definition of this counting requires the description of quasi-Fuchsian groups and their boundary at infinity, done in section~\ref{par:qf}, and their main result (Theorem~\ref{theo:main}) is presented in section~\ref{sec:last}. These results define an entropy-like constant  $E(M,h)$ for minimal surfaces in a Riemannian  manifold  $(M,h)$ of curvature less  than $-1$. The main result of  \textcite{Calegari:2020uo} is to compute it for hyperbolic manifolds, gives bounds in the general case and most notably proves a rigidity result: $E(M,h)=2$ if and only if $h$ is hyperbolic. Altogether, these results  mirror  those for closed geodesics.

When one moves to studying solution of elliptic partial differential equations, for instance minimal surfaces or pseudo-holomophic curves, the situation is different from the chaotic behavior of the geodesic flow. While there is a huge literature about moduli spaces of solutions when one imposes constraints such as homology classes, we do not have that many results describing a moduli space of all solutions: possibly immersed with dense images, in other words to continue the process  for minimal surfaces described in the introduction of \textcite{Gromov:1991uv} for geodesics: {\em if one wishes to understand closed geodesics not as individuals but as members of a community one has to look at all (not only closed) geodesics in X which form an $1$-dimensional foliation of the projectivized tangent bundle.}

The presentation of these notes shifts around the ideas used in \textcite{Calegari:2020uo} and follows more directly the philosophy introduced in 
\textcite{Gromov:1991uv}. We  focus on the construction of such a moduli space
---that we call  the {\em phase space of stable minimal surfaces}--- and its
topological  properties ---see section~\ref{par:phasespace} and
Theorem~\ref{theo:phasespace}. These properties are a rephrasing of
Theorem~\ref{coro:KK} about quasi-isometric properties of stable minimal
surfaces, relying on results of  \textcite{Seppi:2016ut} and a ``Morse type Lemma'' argument  by \textcite[Theorem 3.1]{Calegari:2020uo}. This space is  the analogue, in our situation,  of the  geodesic flow and   the $\mathbb R$-action is  replaced by an $\sld$-action. 

Then we move to studying  $\sld$-invariant measures on this phase space and show they are related to what we call {\em laminar currents} which are the analogues in our situation of geodesic currents ---see \textcite{Bonahon:1997tl}. The main result is now  an equidistribution result in this situation: Theorem~\ref{theo:equi}. This theorem follows from the techniques of the proof of Surface Subgroup Theorem using the presentation given in \textcite{Kahn:2018wx}.

This Equidistribution Theorem and the construction of the phase space allows, by comparing the counting with respect to area and the genus ---as in  \cite{Kahn:2010uo}--- to proceed quickly to the  proof of the results of \textcite{Calegari:2020uo} when, for the rigidity result, we assume that $h$ is close enough to a hyperbolic metric.

The whole article of \textcite{Calegari:2020uo} mixes beautiful ideas from many subjects, adding to the mix of  topology, Riemannian geometry, geometric group theory and ergodic theory used in the study of the geodesic flow,  a pinch of geometric analysis. The approach given in these notes is not just to present the proof but also to take the opportunity to tour some of the fundamental results in these various mathematics\footnote{The introduction of \textcite{Calegari:2020uo} also addresses minimal hypersurfaces in higher dimension that we do not discuss here}. We take some leisurely approach and explain some of the main results and take  the time to give a few simple proofs and elementary discussions:  the clever proof of Thurston showing that there are only finitely many surface groups  of a given genus in the fundamental group of a hyperbolic manifold, the discussion of stable minimal surfaces,  the geometric analysis trick that derives from a rigidity result (here the characterization of the plane as the unique stable minimal surface in $\mathbb R^3$) some compactness results (proposition~\ref{pro:F-Sch}).

During the preparation of these notes, I benefited from the help of many colleagues, as well as the insight of the authors. I want to thank them here for their crucial input: Dick Canary, Thomas Delzant, Olivier Guichard, Fanny Kassel,  Shahar Mozes, Hee Oh, Pierre Pansu, Andrea Seppi, Jérémy Toulisse and  Mike Wolf.

\section{Counting geodesics and equidistribution}\label{sec:par}
When $(M,h)$ is a negatively curved manifold, there is a one-to-one correspondence between conjugacy classes of elements of $\pi_1(M)$ and closed geodesics.  Even though there are infinitely many closed geodesics,  we can count them ``asymptotically''. Equivalently, this will give an asymptotic  count of  the conjugacy classes of elements of $\pi_1(M)$, or to start a point of view that we shall pursue later, the set of free homotopy classes of maps of $S^1$ in $M$. 

We review here some important results that will be useful in our discussion and serve as a motivation.

\subsection{Entropy and asymptotic counting of geodesics} Let $(M,h)$ be a closed manifold of negative curvature. Fixing a positive constant $T$, there are only finitely many   closed geodesics of length less than $T$. Let us define
$$
\Gamma_h(T)\defeq\{\hbox{geodesic }\gamma\mid \operatorname{length}(\gamma)\leq T\}\ .
$$
The  following limit, when it is defined,
$$
{\rm h}_{top}(M,h)\defeq \lim_{t\to\infty}\frac{1}{T}\log\left(\sharp \Gamma_h(T) \right)\  ,
$$
is called the  {\em topological entropy of $M$}.   We will see it is always defined in negative curvature. It measures the exponential growth of the number of geodesics with respect to the length. The topological entropy is related to the {\em volume entropy of $M$} defined by 
$$
{\rm h}_{vol}(M,h)=\liminf_{t\to\infty}\frac{1}{R}\log\bigl(\Vol(B(x,R)) \bigr)\ ,
$$
where $B(x,R)$ is the ball of radius $R$ in the universal cover $\tilde M$ of $M$, $x$ any point in $\tilde M$. The volume entropy does not depend on the choice of the point $x$ and we have
\begin{theo}\label{theo:entropy} Let $(M,h)$ be a closed negatively curved manifold.
\begin{enumerate}
	\item The topological  ${\rm h}_{top}(M,h)$ is well-defined. When $h_0$ is hyperbolic\footnote{that is when the curvature is constant and equal to $-1$}, $${\rm h}_{top}(M,h_0)=\dim(M)-1\ .$$	
	\item We have 
$$
{\rm h}_{top}(M,h)={\rm h}_{vol}(M,h)=\lim_{R\to\infty}\frac{\log\left(\sharp\{\gamma\in\pi_1(M)\mid d_M(\gamma.x,x)\leq R\}\right)}{R}\ .
$$

\end{enumerate}
\end{theo}
The first item is a celebrated result by \textcite{Bowen:1972} and \textcite{Margulis:1969ve}. The second item is due to  \textcite{Manning:1979vk}.

\subsubsection{Rigidity of  the entropy}
We have several  rigidity theorems for the entropy.
 First in the presence of an upper bound on the curvature, a metric on closed manifold has curvature less than $-1$, then   	 $$
	{\rm h}_{vol}(M,h)\geq \dim(M)-1\ ,
	$$
	 with equality if and only if $h$ is hyperbolic. 
For deeper results in the presence of upper bounds on the curvature, see \textcite{Pansu:1989ug} and   \textcite{Hamenstadt:1990tx}.
 As a special case  of \textcite{Besson:1995um}, we have, when we drop the condition on the curvature 
\begin{theo}
	Let $(M,h_0)$ be a hyperbolic manifold of dimension $m$ and $h$ another metric on $M$,  then
	$$
	{\rm h}_{vol}(M,h)^m\Vol(M,h)\geq {\rm h}_{vol}(M,h_0)^m\Vol(M,h_0)\ .
	$$
	The equality implies that $h$ has constant curvature.
\end{theo}
In this expos\'e, we will only use the case of $m=2$, which is due to \textcite{Katok:1982uv}.

\subsection{Equidistribution} This asymptotic counting has a counterpart
called {\em equidistribution}. Let us first recall that geodesics are
solutions of some second order differential equation, and we may as well
consider non closed geodesics in the Riemannian manifold $M$. Let us consider the {\em phase space} $\mathcal G$ of this equation as the space of maps $\gamma$ from $\mathbb R$ to $M$, where $\gamma$ is an arc length parametrized solution of the equation. The precomposition by translation gives a right action by $\mathbb R$, and thus $\mathcal G$ is partitioned into {\em leaves} which are orbits of the right action of $\mathbb R$. The space $\mathcal G$ canonically identifies with the unit tangent bundle~$\mathsf U M$ by the map $\gamma\mapsto(\gamma(0),\dot\gamma(0))$, and the above $\mathbb R$-action corresponds to the action  of the {\em geodesic flow}.

We may thus associate to each closed orbit~$\gamma$ of length~$\ell$ a unique probability measure~$\delta_\gamma$ on  $\mathcal G=\mathsf U M$ supported on $\gamma$, $\mathbb R$-invariant and so that for any function on~$\mathsf U M$ 
$$
\int_{\mathsf U M} f {\rm d}\delta_\gamma\defeq\frac{1}{\ell}\int_0^\ell f(\gamma(s)){\rm d}s\ .
$$
When $M$ is hyperbolic, another natural and  $\mathbb R$-invariant probability measure  comes from the left invariant $\mu_{Leb}$ measure (under the group of isometries) in the universal cover.

The next result is intimately related to Theorem~\ref{theo:entropy} and also due to \textcite{Bowen:1972} and \textcite{Margulis:1969ve}.
\begin{theo}\label{theo:equi} Assume $(M,h_0)$ is hyperbolic, then
$$ \lim_{T\to\infty} \frac{1}{\sharp \Gamma_{h_0}(T)}\sum_{\gamma\in  \Gamma_{h_0}(T)}\delta_\gamma=\mu_{{Leb}}\ .$$
\end{theo}

\section{Totally geodesic submanifolds of higher dimension}\label{sec:tg}

As a first attempt of generalization, it is quite tempting to understand what happens to {\em totally geodesic} submanifolds of higher dimension, where by totally geodesic we mean complete  and such that  any geodesic in the submanifold is a  geodesic for the ambient manifold. 

\subsection{Closed totally geodesic submanifolds are rare}
One easily constructs by arithmetic means hyperbolic manifolds with infinitely many closed totally geodesic submanifolds, however this situation is exceptional and we have, as a special case of a beautiful recent theorem by \textcite{Margulis:2020wp} -- generalized in \textcite{Bader:2019wu}:

\begin{theo}
	If  a closed hyperbolic 3-manifold $M$ contains infinitely many closed totally geodesic subspaces of dimension at least  $2$, then $M$ is arithmetic.
\end{theo}
Thus, having infinitely many closed totally geodesic submanifolds is quite rare for hyperbolic $3$-manifolds, and an asymptotic counting as we defined for geodesics does not yield interesting results in general -- see however \textcite{Jung:2019tq} for a result in the arithmetic case and a general upper bound in \textcite[Corollary 1.12]{Mohammadi:2020wi}.

\subsection{The set of pointed totally geodesic spaces and Shah's Theorem}\label{par:Ratner}

Let $(M,h)$ be an oriented  Riemannian $3$-manifold and $G(M)$ be the bundle over $M$ whose fiber at a point $x$ is the set of oriented $2$-planes in the tangent space at $x$. Every surface $S$ in $M$ then has a {\em Gauß lift} $G(S)$ in $G(M)$ which consists of the set of tangent spaces to~$S$.

\subsubsection{Totally geodesic hyperbolic planes, the frame bundle and the $\psld$-action}

 When $h$ is hyperbolic, the space $G(M)$ has a natural foliation $\mathcal F$ whose leaves are Gauß lifts of immersed totally geodesic hyperbolic planes. We can thus interpret $G(M)$ as the space of pointed totally geodesic planes, or equivalently as the set of (local) totally geodesic embeddings of $\hh$ into $M$, equipped with the right action by precomposition of~$\psld$.

Let  $F_{h_0}(M)$ be the {\em frame bundle} over $M$ whose fiber at $x$ is the set of oriented orthonormal frames in the tangent space of $x$. We have a natural fibration
$$
S^1\to F_{h_0}(M)\to G(M)\ .
$$
The choice of a frame at a point $x$ in $\hhh$ identifies $\iso(\hhh)$, the group of orientation preserving isometries of $\hhh$, with  $\pslc$. Thus $F_{h_0}(\hhh)$ is interpreted as the space of isomorphisms of  $\iso(\hhh)$ with $\pslc$ and as such carries commuting actions of $\iso(\hhh)$ on the left by postcomposition and $\pslc$ on the right by precomposition. Then the foliation  of $F_{h_0}(\hhh)$ by the orbits of the right action of the subgroup $\psld$ of $\pslc$,  projects to the foliation $\mathcal F$ of $G(\hhh)$ that we just described, with the corresponding action of $\psld$.  

\subsubsection{Shah and Ratner's Theorem}
From this interpretation of the $\psld$-action on $G(M)$, a theorem by \textcite{Shah:1991wr}, which is now also  a consequence of the celebrated theorem by \textcite{Ratner:1991wl}, gives:

\begin{theo}\label{theo:Ratner}
Let $(M,h_0)$ be a closed hyperbolic $3$-manifold. Then any  orbit of the $\psld$-action on $F_{h_0}(M)$ is either dense or closed. Moreover any closed  set invariant by $\psld$ in $F_{h_0}(M)$ is either everything or a finite union of closed orbits. 
\end{theo}

Recent work by \textcite[Theorem 11.1]
{McMullen:2017ub} give results in the non closed case, see also~\textcite{Tholozan:2019uw}. 

We state this result as part of our promenade in the subject and will not use it in the proof, as opposed to \textcite{Calegari:2020uo}. However, we will use as special case of the measure classification theorem of \textcite{Ratner:1991tu}.

\begin{theo}\label{theo:Ratner2}
Let $(M,h_0)$ be a closed hyperbolic $3$-manifold. Then any  ergodic $\sld$-invariant measure on $F_{h_0}(M)$ is $\mu_{Leb}$ or is supported on a closed leaf in $F_{h_0}(M)$.
\end{theo}
In particular, there are only countably many ergodic $\sld$-invariant measures.  For a short and accessible proof of Ratner's theorem in the context of $\sld$-action see \textcite{Einsiedler:2006wz}.

We only need the following corollary that may have a direct proof.
\begin{coro}\label{coro:Ratner3}
	Let $\mu$ be an $\sld$-invariant probability measure on $F_{h_0}(M)$. Let~$p$ be the projection of $F_{h_0}(M)$ to $M$. Assume that $p_*\mu$ is \textup{(}up to a constant\textup{)} the volume form on $M$, then $\mu=\mu_{Leb}$. 
\end{coro}

\subsection{In variable curvature} For a generic metric,  there are no totally geodesic surfaces, not even locally. Thus, even non closed totally geodesic surfaces are rare in variable curvature. We actually explain now a more precise result of \textcite{Calegari:2020uo} which implies that if a negatively curved manifold has too many  hyperbolic planes, then it is hyperbolic.

We start by discussing briefly the boundary at infinity of negatively curved manifolds.
\subsubsection{The boundary at infinity and circles}\label{par:circ}

In the Poincaré ball model, $\hhh$ is the interior of a ball. The boundary of the ball is denoted $\partial_\infty \hhh$ and carries an action of the isometry group $\iso(\hhh)$ of $\hhh$, which is isomorphic to $\pslc$. Under this isomorphism, $\partial_\infty\hhh$ identifies as a homogeneous space with  $\mathbf{CP}^1$.

Observe that the choice of real plane $P$ in $\mathbb C^2$ defines a {\em circle} in $\mathbf{CP}^1$ which is the set of complex lines intersecting non trivially the plane $P$. In the ball model of $\hhh$, these circles are boundaries of hyperbolic planes totally geodesically embedded in $\hhh$ .

The boundary at infinity $\partial_\infty\hhh$ has an intrinsic definition as the set of equivalence classes of oriented geodesics which are {\em parallel at infinity}, where the equivalence is defined as
$$
\gamma_1\sim\gamma_2\ \ \hbox{ if and only }\ \ \ \limsup_{t\to\infty}d(\gamma_1(t),\gamma_2(t))<\infty\ .$$
This notion also makes sense in the case of a nonpositively curved  simply connected manifold $\tilde M$ and allow us to define the notion of  
the boundary at infinity $\partial_\infty\tilde M$.

We recall here briefly that  $\tilde M\sqcup\partial_\infty\tilde M$ admits a topology and becomes so a compactification of $\tilde M$ as a closed ball. Moreover, if $\tilde M$ is the universal cover of a closed manifold that admits a hyperbolic metric, we have an identification of $\partial_\infty\tilde M$ with $\partial_\infty\hhh$.

If $M$ has dimension at least  3, by Mostow rigidity, the above identification of $\partial_\infty\tilde M$ with $\mathbf{CP}^1$ is unique up to the action of $\pslc$ and thus circles make  perfect sense in~$\partial_\infty \tilde M$.

\subsubsection{A characterization of hyperbolic metrics}

\begin{prop}\label{pro:manyhyp}
	Assume that the curvature of the closed  $3$-manifold $M$ is less than $-1$. Assume that every circle at infinity bounds  a totally geodesic hyperbolic surface. Then $M$ is hyperbolic.
\end{prop}
\begin{proof}[Sketch of a proof] Since the curvature is negative, a circle at
  infinity cannot bound more than one totally geodesic surface. Let $\mathcal
  T$ be the space of 
   triples of pairwise distinct elements
  in $\bhhh$. Every such triple $\tau$ defines a unique circle $\gamma$ and thus a unique totally geodesic oriented hyperbolic plane $\mathbf H$ in $\tilde M$, the universal cover of $M$. Since $\gamma$ is the boundary at infinity of $\mathbf H$, a triple of points $\tau=(a,b,c)$ define a  frame $(x,u,v)$ where $x$ is a point in $\mathbf H$ and $(u,v)$ an orthonormal basis of  $\T_x\mathbf H$ by the following procedure: $x$ is on the geodesic joining $a$ to $c$, $u$ is the vector tangent at $x$ in the direction of $b$ and $v$ the tangent vector in the direction of $c$. Let us define
$$
\beta(\tau)=(x,u,v,n)\ ,\hbox{ such that $(u,v,n)$ is an oriented orthonormal basis of $\T_x\tilde M$}\ .
$$	
Then $\beta$ is $\pi_1(M)$-equivariant from $F_{h_0}(\hhh)$ to $F_{h}(M)$. One can show that $\beta$~has degree~$1$ and is surjective. It follows in particular that there is a totally geodesic hyperbolic plane through any tangent plane in~$M$. Thus $M$~has constant curvature~$-1$.
\end{proof}

\section{Surface subgroups in hyperbolic $3$-manifolds}\label{par:qf}
 
Let us move to topological questions. The natural question is, after  we have spent some time in the first sections studying homotopy classes of circles in negatively curved manifolds,  to understand conjugacy classes of fundamental groups of surfaces in $3$-manifolds. We will spend some time recalling classical  facts about quasi-Fuchsian surface groups, the surface subgroup theorem by Kahn and Markovi\'c, and finally asymptotic counting of those surface subgroups by the genus.  

To be explicit, a {\em surface group} is the fundamental group of a compact
connected orientable surface of genus greater than or equal to~$2$. 

Any such group can be represented as a {\em Fuchsian subgroup}, that is a discrete cocompact subgroup of  the isometry group $\psld$ of the hyperbolic plane $\hh$. 

We concentrate our discussion first on discrete surface subgroups of $\iso(\hhh)$.

\subsection{Quasi-Fuchsian groups and quasicircles} Let $S$ be  a closed
connected oriented surface $S$ of genus greater than $2$. Let $\rho_0$ be a faithful representation of the fundamental group $\pi_1(S)$ in  $\iso(\hh)$ whose image is a cocompact lattice $\Gamma$.  Seeing $\hh$ sitting as a geodesic plane in $\hhh$,  gives rise to an embedding of $\iso(\hh)$ in $\iso(\hhh)$. The corresponding morphism of $\Gamma<\iso(\hh)$ in $\iso(\hhh)$ is called a {\em Fuchsian representation} and its image a {\em Fuchsian group}\footnote{we warn the reader our definitions are slightly non standard here}.

We saw in paragraph~\ref{par:circ} that such a Fuchsian group preserves a circle in $\partial_\infty\hhh$. This motivates the following definitions.

\begin{defi} 
  \begin{enumerate}
  \item 	a {\em quasi-Fuchsian representation} is a morphism $\rho$ from a cocompact lattice $\Gamma$ of $\iso(\hh)$ in $\iso(\hhh)$, such that there exists a continuous injective map $\Lambda$ from $\partial_\infty\hh$ to  $\partial_\infty\hhh$ which is $\rho$-equivariant. 
  \item  The map $\Lambda$ is called the {\em limit map} of the quasi-Fuchsian morphism.
  \item a {\em quasi-Fuchsian  group} is the image of a quasi-Fuchsian representation.
  \item a {\em quasi-Fuchsian  manifold} is the quotient of the hyperbolic space by a quasi-Fuchsian group.
  \item The {\em limit set $\partial_\infty\Gamma$} of a \qf\ group $\Gamma$ is the image of its limit map.
  \end{enumerate}
\end{defi}
In this definition, observe that the choice of a quasi-Fuchsian representation
depends on the choice of a realization of the surface group as a lattice in $\psld$. Similarly, the limit map of a \qf\  representation depends on the choice of a lattice in $\psld$, while $\partial_\infty\Gamma$ only depends on the \qf\ group $\Gamma$.

Quasi-Fuchsian manifolds are not compact: they are homeomorphic to $S\times \mathbb R$, if the quasi-Fuchsian group is isomorphic to $\pi_1(S)$. However, they keep some cocompactness feature:

\begin{prop}
	Let $\partial_\infty\Gamma$ be the limit set of a quasi-Fuchsian surface group $\Gamma$ and $\Env(\partial_\infty\Gamma)$ be the convex hull  \textup{(}for instance in the projective Klein model\textup{)} of $\partial_\infty\Gamma$ in $\hhh$. Then
	\begin{enumerate}
		\item  the distance to $\Env(\partial_\infty\Gamma)$ is convex\footnote{A function is {\em convex} its restriction to any geodesic is convex. In nonpositive curvature, any distance to a convex set is convex},
		\item the group $\Gamma$ acts cocompactly on $\Env(\partial_\infty\Gamma)$.
\end{enumerate} 	
	 The quotient  $\Env(\partial_\infty\Gamma)/\Gamma$ is called the {\em convex core} of the quasi-Fuchsian manifold associated to $\Gamma$. 
\end{prop}
Quasi-Fuchsian groups are plentiful, and in particular 

\begin{prop}
	Any small deformation of a Fuchsian group is quasi-Fuchsian.
\end{prop}

The limit map of a quasi-Fuchsian group has many remarkable properties:

\begin{prop}
		Given the limit map $\Lambda$ of a quasi-Fuchsian group, there exists a constant $K$, such that for any quadruple of pairwise distinct points $(x,y,z,t)$ 
	in $\partial_\infty\hh\simeq\bf{RP}^1$, then
\begin{equation}
		\bigl\vert [\Lambda(x),\Lambda(y),\Lambda(z),\Lambda(w)]\bigr\vert\leq K	\bigl\vert [x,y,z,w]\bigr\vert\ , \label{def:Kquasi}
\end{equation}
where $[a,b,c,d]$ denotes the cross-ratio of the quadruple  $(a,b,c,d)$ in either~$\bf{RP}^1$ or~$\bf{CP}^1$.
\end{prop}
More generally a map $\Lambda$ from 
$\bf{RP}^1$ to  $\bf{CP}^1$ is called {\em $K$-quasisymmetric} if it satisfies inequality~\eqref{def:Kquasi}. The image of a $K$-quasisymmetric map is called a {\em $K$-quasicircle}.
 The above proposition can be strengthened as 
 \begin{prop}
 	The limit map of any quasi-Fuchsian representation  is $K$-quasisymmetric for some $K$. If $K=1$, then the group is actually Fuchsian.
 \end{prop}
 
 Accordingly, a surface group is {\em $K$-quasi-Fuchsian} if it admits a $K$-quasisymmetric limit map. The constant $K$ gives a feeling of  how far a \qf\ group is from being Fuchsian. 

Not all  discrete  surface groups in $\iso(\hhh)$ are quasi-Fuchsian. We shall see an example of that in the next paragraph.

\subsection{Surface subgroups in fundamental groups of closed hyperbolic $3$-manifolds}
Solving  a crucial conjecture 
of Thurston, Kahn and Markovi\'c proved that fundamental groups of closed  hyperbolic $3$-manifolds contain surface groups.  The amazing proof, in \textcite{Kahn:2009wh}, uses mixing and equidistribution of the geodesic flow and we shall have to extract further information from it.

Kahn--Markovi{\'c} surface subgroup theorem  states the existence of many surface groups which are ``more and more'' Fuchsian in some precise way

\begin{theo}{\sc[Kahn--Markovi{\'c} surface subgroup theorem]}	Let $M$ be a closed hyperbolic $3$-manifold. Let $\epsilon$ be any positive constant. Then there exists a quasi-Fuchsian subgroup in $\pi_1(M)$ whose limit map is a $(1+\epsilon)$-quasicircle. 
\end{theo}
This result was explained in a Bourbaki exposé by \textcite{Bergeron2013}. The quantitative  part of the result  plays a crucial role in the proof  of Agol's Virtual Haken Theorem by \textcite{Agol:2013ts}  stating that any hyperbolic $3$-manifold has a finite covering which is a surface bundle over the circle. Quite interestingly, in those manifolds fibering over the circle the fundamental group of the fiber is not quasi-Fuchsian.

To add a little perspective that will come up later, recall that surface groups and fundamental groups of hyperbolic manifolds are prototypes of {\em Gromov-hyperbolic groups}. Gromov has broadened Thurston's conjecture in the following question.

\begin{question}
 Does any  one-ended Gromov-hyperbolic group contain a surface group?
\end{question}

\subsection{Counting surface subgroups}

A classical theorem in geometric group theory says 

\begin{theo}\label{theo:finitely}
	A Gromov-hyperbolic group contains only finitely many conjugacy classes of surface groups of a given genus.
\end{theo}
This result is suggested  in \textcite{Gromov:1987tk},  a proof and a generalization is given in \textcite{Delzant:1995un}. 

The special case of the fundamental group of a hyperbolic $3$-manifold is due to \textcite{Thurston:1997ux} and his beautifully simple proof works in general for fundamental groups of negatively curved manifolds. We prove it in  theorem~\ref{theo:thurs} after  our discussion of minimal surfaces.

For a hyperbolic manifold $M$, let $S(M,g)$ be the number of conjugacy classes of surface subgroups in $\pi_1(M)$ of genus $g$. Thurston already gave a crude estimate of an upper  bound for $S(M,g)$, later on improved by \textcite{Masters:2005vi} and \textcite{Soma:1991tr}. A crucial improvement of this count is made in \textcite{Kahn:2010uo}. 
\begin{theo}\label{theo:KMcount}
	Let $M$ be  a hyperbolic $3$-manifold, then there exist constants~$c_1$ and~$c_2$ so that for~$g$ large enough
	\begin{equation}
		(c_1 g)^{2g}\leq S(M,g)\leq (c_2g)^{2g}\ ,
	\end{equation}
	where $c_2$ only depends on the injectivity radius of $M$.
\end{theo}
The previous upper bound by Masters was of the form  $g^{c_2g}$. To get the lower bound,  it is actually  enough ---and distressing--- to have the existence of one surface group, and counts its covers ---see proposition~\ref{pro:puchta}. However, Kahn and Markovi\'c also have the same estimates in the harder case when one counts commensurability classes.

We deduce,
\begin{coro}\label{cor:KMcount}
\begin{equation}
	\lim_{g\to\infty}\frac{\log(S(M,g))}{2g\log(g)}=1\ .
\end{equation}
\end{coro}
 Kahn and Markovi\'c conjecture a more precise asymptotic\footnote{also for commensurability classes}:
\begin{conj}
	Let $M$ be a hyperbolic $3$-manifold, then there exists a constant $c(M)$ only depending on $M$, so that
	$$
	\lim_{g\to\infty}\frac{1}{g}\left(S(M,g)\right)^{\frac{1}{2g}}=c(M)\ .
	$$
\end{conj}

Observe that this counting is purely topological, the results do not make any reference to the underlying Riemannian structure of the manifold.

\vskip 0.2 truecm
To summarize this discussion, according to \textcite{Kahn:2009wh}, there are many surface groups in $\pi_1(M)$ and we have a purely topological  asymptotic of the growth of   the numbers of those when the genus goes to infinity.  Calegari--Marques--Neves article also addresses the question of counting those subgroups but with a geometric  twist. Before explaining their result, let us review some fundamental results on minimal surfaces.

\section{Minimal surfaces in $3$-manifolds}

In order to recover the flexibility that we lost when considering closed totally geodesic submanifolds, let us now introduce {\em minimal immersions} which will allow us to extend our discussion about geodesics. We first spend some time recalling some basic properties and definition of minimal immersions, before actually addressing the question of counting surface subgroups.

\subsection{Minimal immersions}
 Let $(M,g)$ be a Riemannian manifold. We denote in general by $\vol(h)$, the volume density of a metric $h$. An immersion $f$ from a compact manifold $N$  into   $M$ is  a {\em minimal immersion} if $f$ is a critical point for the {\em volume functional}
$$
\operatorname{Vol}(f)\defeq \int_N \ \vol{f^*g}\ .
$$
More precisely, this means that for any family of smooth deformations $\{f_t\}_{t\in]
-1,1[}$ with $f_0=f$ we have
\begin{equation}
	\left.\frac{{\rm d}}{{\rm d}t}\right\vert_{t=0} \Vol(f_t)=0\ .\label{eq:defmin}
\end{equation}
To the family of deformations $\{f_t\}_{t\in]-1,1[}$ is associated the {\em infinitesimal deformation vector} $\xi$ which  is the section of $f^*(\T M)$ given by 
\begin{equation}
	\xi(x)=f^*\Bigl(\Bigl.\frac{{\rm d}}{{\rm d}t}\Bigr\vert_{t=0} f_t(x)\Bigr)\ .\label{eq:definf}
\end{equation}
\subsubsection{The first variation formula} One can now compute effectively
the left hand side of equation~\eqref{eq:defmin} by a classical computation
which is called the {\em first variation formula}. Let us introduce the second fundamental form $\rm{II}$ which is the symmetric tensor with values in the normal bundle given by 
$$
\rm{II}(X,Y)=p(\nabla_X Y)\,  
$$
where $X$ and $Y$ are tangent vectors to $N$, $\nabla$ is the Levi-Civita connection of $g$ pulled back on $f^*(\T M)$, and $p$ is the orthogonal projection of $f^*(\T M)$ on the normal bundle  of $N$. Then the first variation formula reads 
$$
\left.\frac{{\rm d}}{{\rm d}t}\right\vert_{t=0} \Vol(f_t)= \int_N\braket{\xi\mid H} \vol(f^* g)\ .
$$
where $H$ is {\em the mean curvature vector} defined as  the trace of $\rm{II}$.

Thus being a minimal immersion  is equivalent to the fact that the mean curvature vanishes identically. As an important corollary, we have two useful properties

\begin{coro}\label{cor:subh}
\phantomsection
\begin{enumerate}
	\item The restriction of a convex function to a minimal submanifold is subharmonic.	
	\item The curvature of a minimal surface at a point is less than the ambient curvature of its tangent plane.
\end{enumerate}

\end{coro}

One then defines a {\em minimal immersion} from a (possibly non compact, possibly with boundary) manifold  as one for which the mean curvature vanishes everywhere. Equivalently, one can show that those are the immersions $f$ for which for any variation $\{f_t\}_{t\in]-1,1[}$ with $f_0=f$,  constant on the boundary as well as outside a bounded  open set $U$,  we have

$$
\Bigl.\frac{{\rm d}}{{\rm d}t}\Bigr\vert_{t=0} \int_U\vol(f_t^* g)=0\ .
$$
When the dimension of~$N$ is~$1$, minimal immersions are  exactly parameterizations of geodesics.

\subsubsection{The second variation formula} The misleading terminology ``minimal immersions'' or ``minimal surfaces'' tends to  suggest that minimal surfaces are not only critical point of the area functional but actual minima. This not always the case.

In order to understand whether the immersion is actually a local minimum of the volume functional, we need ---as in the case for geodesics--- to study the  {\em second variation formula of the volume}. 

Let us assume for simplicity that the source is compact, since our goal is only to present the subject.

Let us denote by $\xi$ an infinitesimal variation of $f$ as in equation~\eqref{eq:definf}. We may as well assume that this infinitesimal variation is normal since tangent deformation do not affect the volume. Then the second variation formula is given by
\begin{equation}
	{\rm D}^2_f\Vol(\xi,\xi)\defeq\Bigl.\frac{{\rm d^2}}{{\rm d}t^2}\Bigr\vert_{t=0} \int_N\vol(f_t^* g)= 2\int_N (R_\xi+a_\xi -b_\xi)\vol(f^* g) \ ,\label{eq:2vf}
\end{equation}
where $R_\xi$, $a_\xi$ and $b_\xi$ are the trace ---with respect to the induced metric--- of the symmetric tensors defined by respectively
\begin{align*}
	R_\xi(X,Y)&\defeq\braket{R(\xi,X)\xi\mid Y}\ ,\\
	a_\xi(X,Y)&\defeq\braket{p(\nabla_X\xi)\mid p(\nabla_Y\xi)}\ , \\
	b_\xi(X,Y)&\defeq\braket{B(X)\xi\mid B(Y)\xi}\ ,
\end{align*}
where $X$ and $Y$ are tangent vectors and $\xi$ is normal, $R$ is the curvature tensor of the Levi-Civita curvature $\nabla$ of the ambient manifold, $B$ is the {\em shape operator} defined by
$$
\braket{B(X)\xi\mid Y}=\braket{\rm{II}(X,Y)\mid \xi}\ .
$$ 
Since our ultimate goal is to understand the sign of ${\rm D}^2\Vol$, we now comment on the sign of these quantities: 
\begin{enumerate}
\item $a_\xi$ is nonnegative.
	\item $b_\xi$ is nonnegative, but vanishes when the submanifold is {\em totally geodesic}.
	\item $R_\xi$ is nonnegative when the ambient curvature is nonpositive. When the ambient manifold is hyperbolic, $R_\xi=2\Vert\xi\Vert^2$.
\end{enumerate}
In particular, when the ambient curvature is nonpositive and the submanifold is totally geodesic then $D_f^2\Vol(\xi,\xi)$ is  nonnegative and the minimal immersion is a local minimum. This covers for instance the case of geodesics in nonpositive curvature.

However, in general one cannot expect that just controlling the sign of the curvature would guarantee that the minimal immersion is an actual local minimum. Nevertheless, we shall see that under some other additional assumptions the minimal immersion will be a local minimum. We now introduce the standard terminology:

\begin{defi}
A minimal immersion $f$ is {\em stable}\footnote{The terminology is unstable here: some people call this condition {\em semistable}} if for any compactly supported infinitesimal deformation $\xi$, ${\rm D}^2_f\Vol(\xi,\xi)\geq 0$.
\end{defi}
Thus a totally geodesic surface in a nonpositively curved manifold is stable.
A famous result proved independently  in
\textcite{Pogorelov:1981uq}, \textcite{Carmo:1979uo}, and \textcite{Fischer-Colbrie:1980vl} states
\begin{theo}\label{theo:CP}
	The plane is the only stable embedded complete minimal surface in~$\mathbb R^3$.
\end{theo}

As a standard phenomenon, we will explain later on, that such a rigidity result implies a compactness property for the space of stable minimal surfaces (proposition~\ref{pro:F-Sch}).

\subsection{Minimal surfaces in $3$-manifolds} Let us now focus on minimal immersions of surfaces in $3$-manifolds. The volume is then called the {\em area} of the immersion,
and denoted by $\operatorname{Area}(f)$.

Two important results  by \textcite{Schoen:1979} and \textcite{Sacks:1982} guarantee the existence of minimal surfaces from some topological data. Here is a special case of their result.

Say a continuous map between two  connected manifolds is {\em incompressible} if it is injective at the level of fundamental groups.

\begin{theo}\label{theo:SY}
Let $f$ be a continuous incompressible map from a closed surface to a compact negatively curved  $3$-manifold. Then there exists a minimal immersion, homotopic to $f$, which is minimal and achieves the minimum of the area amongst all possible maps homotopic  to~$f$.
\end{theo}

In particular, surface groups in fundamental groups of compact hyperbolic manifolds can be represented by minimal surfaces, albeit not necessarily uniquely as we shall see.

As an application,  let us now give a hint of the proof by \textcite{Thurston:1997ux} of theorem~\ref{theo:finitely} in the case of the fundamental groups of a hyperbolic manifold as was explained to us by Delzant. \begin{theo}\label{theo:thurs}
	The fundamental group of a hyperbolic manifold  only contains finitely many conjugacy classes of surface groups of a given genus.
\end{theo}
\begin{proof} Let $S$ be a minimal surface in $M$ representing a surface group.
	The curvature of the minimal surface $S$ is bounded from above by the curvature of $M$, and thus the area $\Area(S)$ of $S$ is bounded from above by $4\pi(g-1)$, where $g$ is the genus of $S$.
	
	Moreover since the surface is incompressible, the injectivity radius $i_S$ of $S$ is bounded from below by the injectivity radius  $i_M$ of $M$. Let then $a(i_M)$ be a lower bound of the area of a ball of radius $i_M/2$ in $S$ and observe that by comparison theorems, we can have an explicit formula for $a(i_M)$ in terms of $i_M$. 
	
	Thus we can cover $S$ by $\Area(S)/a(i_M)$ balls of radius $i_M/2$. Hence $\pi_1(S)$ is generated by curves of length (in $S$) less than $2i_M\Area(S)/a(i_M)$ and hence less than $8\pi i_M (g-1)/a(i_M)$. The same holds {\it a fortiori} for the length of those curves in $M$. This implies that there is only finitely many possibilities for conjugacy classes of surface groups.
\end{proof}
Observe that Thurston originally used pleated surfaces rather than minimal ones: we only use the fact that the surface representing the surface group has curvature no greater than $-1$.

\subsubsection{Back to quasi-Fuchsian manifolds}
The work of Schoen and Yau carries on immediately in the context of quasi-Fuchsian manifolds, due to the existence of a convex hull which traps minimal surfaces:

\begin{prop}\label{pro:trap} Let $M$ be a $3$-manifold of curvature less than $-1$. Let $\rho$ be a representation of $\pi_1(S)$ in the isometries of $M$. Let $C$ be a convex set in $M$ invariant by the action of $\rho(\Gamma)$. 
	Let $f$ be a  minimal immersion of the universal cover $\tilde S$ of $S$,  equivariant under $\rho$, then $f(S)$ is a subset of $C$.
		\end{prop}
\begin{proof} The distance function to the convex set $C$ is convex and strictly convex for positive values. Hence by corollary~\ref{cor:subh}, its pullback on $S$ is strictly subharmonic for positive values and $\pi_1(S)$-invariant, hence vanishes identically.
\end{proof}

However the number of those surfaces is not {\it a priori} bounded, as it follows from results of  \textcite{Anderson:1983cm} and \textcite{Huang:2015wl}

\begin{theo}
	 For any given positive integer $N$, there exists a quasi-Fuchsian manifold that contains at least $N$ distinct \textup{(}closed, incompressible and embedded\textup{)} minimal surfaces.
\end{theo}

For a survey about the use of minimal surfaces in $3$-manifolds, see \textcite{Hass:2005tr}.

\subsection{Compactness results} 

The set of minimal surfaces enjoy compactness properties. In particular we have the standard fact valid in all dimensions.
\begin{prop}\label{pro:compac}
	Let $\seqm{(M_m,h_m,x_m)}$ be a sequence of pointed Riemannian manifolds converging to a Riemannian manifold $(M_\infty,h_\infty,x_\infty)$. For each $m$, let $S_m$ be a complete minimal surface  without boundary in $M_m$, so that $x_m$ belongs to $S_m$. Assume that, for every $R$, the second fundamental form of $S_m$ is bounded independently on $m$, on every ball in $S_m$ containing $x_m$ and radius $R$.

	Then the sequence of pointed minimal surfaces $\seqm{(S_m,x_m)}$ converges uniformly on every compact  to a pointed  minimal surface $(S_\infty,x_\infty)$ in $M_\infty$
\end{prop}

The following is essentially contained in \textcite{Fischer-Colbrie:1980vl} and is a consequence of Theorem~\ref{theo:CP}. Let $\lambda$ be the function on a minimal surface defined as the positive eigenvalue of the shape operator.

\begin{prop}\label{pro:F-Sch}
	Let $M$ be a $3$-manifold with a metric $h$ with curvature bounded from above by $-1$. Then there is a positive $K$ only depending on $M$, such that for any  stable minimal disk $D$ embedded  in the universal cover  $\tilde M$, we have $\lambda(D)\leq K$. 
\end{prop}

We sketch a proof to emphasize  the standard philosophy in geometric analysis that a rigidity result yields a compactness result.
\begin{proof}[Sketch of a proof]
 Let us give a proof by contradiction and assume that there exists a sequence of stable complete minimal disks $\seq{D}$ and a point $x_m$ in $D_m$ so that $\seqm{\lambda_m(x_m)}$ goes to infinity. Here we denote by $\lambda_m$ the function $\lambda$ on $D_m$. We can assume using the cocompact group $\pi_1(M)$ in $\iso(\tilde M)$ that $x_m$ lies in a compact fundamental domain for $\pi_1(M)$.
 
 The optimum would be to find a point~$x_m$ in~$D_m$ where $\lambda_m$~achieves its maximum value~$\Lambda_m$, which happens for instance if the disk projects to a closed surface in~$M$. However, since this is not necessarily the case, we use a classical trick in geometric analysis. 
 
 Let then  $K_m$ be the maximum of $\lambda_m$ on the ball of center $x_m$ and radius $10$. 
  According to  the {\em $\Lambda$-maximum lemma} as in \textcite[Paragraph 1.D]{Gromov:1991uy}, and assuming $\lambda_m(x_m)\geq 1$  there exists $y_m$,  such that	
  \begin{equation}
	   \Lambda_m\defeq \lambda_m(y_m)\geq \sup\Bigl\{K_m,\frac{1}{2}\lambda_m(z)\mid d(z,y_m)\leq  \frac{1}{2\sqrt{\lambda_m(y_m)}}\Bigr\}\ .\label{eq:lambda}
	 \end{equation}	
Observe that $\seq{\Lambda}$ also goes to infinity. 
We now consider  the metric  $h_m=\Lambda_mh$ on~$M$, 
associated to a distance $d_m$, and observe that $(\tilde M,h_m,y_m)$ converges smoothly on every compact to a Euclidean space.  The new eigenvalue function $\tilde\lambda_m$ is now equal to~\mbox{$\lambda_m/\Lambda_m$}.

Thus we obtain from the assertion~\eqref{eq:lambda}, that  
$$
\tilde\lambda_m(z)\leq 2, \hbox{ if } d_m(z,y_m)\leq \frac{\sqrt{\Lambda_m}}{2}\ .
$$
Thus the sequence of minimal surface $\seqm{(S_m,y_m)}$ has bounded second
fundamental form on larger and larger balls, and thus, by
proposition~\ref{pro:compac},  the sequence  converges (on every compact)  to
a minimal surface $(S_\infty,y_\infty)$ in $(M,h_\infty,y_\infty)$. Now
$S_\infty$ is a stable minimal surface in the Euclidean $3$-space  $(M,h_\infty)$. By Theorem~\ref{theo:CP},  $S_\infty$ is a plane and thus  its  second fundamental form  is zero. Hence,  $\tilde\lambda_m$ converges to zero uniformly on every compact and this contradicts $\tilde\lambda_m(y_m)=1$. This proves by contradiction that $\seq{\Lambda}$ is  bounded. It follows that $\seq{K}$ --- and in particular $\seqm{\lambda_m(x_m)}$ --- is bounded. \end{proof}

\subsection{Almost Fuchsian minimal surfaces, Uhlenbeck's result and the asymptotic Plateau problem}
The examples constructed by Anderson, then by Huang  and Wang are far from being Fuchsian but  the situation improves when we are close to being Fuchsian.

\subsubsection{Almost Fuchsian minimal surfaces}
Let us go back to the second variation formula~\eqref{eq:2vf} for surfaces in the case of hyperbolic $3$-manifolds. We saw that the Hessian of the volume at a minimal immersion $f$ is given by \begin{equation}
	{\rm D}^2_f\Vol(\xi,\xi)= 2\int_N (R_\xi+a_\xi -b_\xi)\vol(f^* g)\ ,
\end{equation}
where $a_\xi\geq 0$, $R_\xi=2\Vert\xi\Vert^2$ and 
$$
b_\xi=\operatorname{Trace}\braket{B(X)\xi\mid B(Y)\xi}=2\lambda^2\Vert\xi\Vert^2\ , 
$$
where $\lambda$ is the positive eigenvalue of $B$. Thus if we assume that $\lambda<1$ we can guarantee that $b_\xi\leq R_\xi$ and thus that ${\rm D}^2\Vol(\xi,\xi)>0$, for a non vanishing $\xi$, hence that $S$ is stable. 

This suggest the following definition, where the  term {\em almost Fuchsian} was coined by \textcite{Krasnov:2007wd}.

\begin{defi}
	\begin{enumerate}
	\item A {\em  nearly geodesic minimal surface} is a complete minimal surface $S$ in $\hhh$ with  $\lambda(S)<1$. If the nearly geodesic surface is invariant under a quasi-Fuchsian group, we say the quasi-Fuchsian group is {\em almost Fuchsian} and the nearly geodesic surface almost Fuchsian.	
		\item An {\em almost Fuchsian manifold} is a quasi-Fuchsian manifold that contains an almost Fuchsian minimal surface.
		\end{enumerate}
		\end{defi}
Using our freshly minted terminology, we can rephrase the previous discussion as the first part of the proposition
\begin{prop}
A nearly geodesic minimal surface in a hyperbolic $3$-manifold  is stable.
\end{prop}

Then, as suggested by this stability result, \textcite{Uhlenbeck:1983wl} proved the following, which is a simple application of the maximum principle.

\begin{theo}\label{theo:uhl}
An almost Fuchsian manifold contains a unique minimal embedded incompressible surface, which is then stable.
\end{theo}

\subsubsection{Asymptotic Plateau problem} Let us quit the realm of equivariant minimal surfaces. As a special case of a theorem of  \textcite{Anderson:1982vw} we have 
\begin{theo}
	Given any embedded circle $C$ in $\bhhh$, there exists a minimal embedded surface in $\hhh$ bounded by $C$ in the Poincaré ball model.
\end{theo}
Let us then say that
\begin{defi}
An embedded minimal surface $S$ is solution of the {\em asymptotic Plateau problem} defined by the embedded circle $C$ in $\bhhh$ if in the ball model the closure $\overline{S}$ of $S$ is 
$
S\sqcup C.
$	
Alternatively we say that $C$ is the {\em boundary at infinity} of $S$, that $S$ is bounded by $C$ and write $C=\partial_\infty S$.
\end{defi}

One naturally hopes there should be a correspondence between $K$-quasicircle and almost Fuchsian minimal surfaces. This is indeed  obtained as a consequence of a theorem of  \textcite[Theorem A]{Seppi:2016ut}, while the second part follows from an extension of \textcite{Guo:2010wk}, where the result is only stated for almost Fuchsian surfaces.

\begin{theo}\label{theo:Seppi}
	There exist  constants  $K_0$ and  $C_0$ such that if $S$ is a complete embedded minimal surface  whose boundary at infinity is a $K$-quasicircle, with $K$ less than $K_0$,  then
	$$
	\lambda(S)\leq C_0 \log(K)\ .
	$$
	Conversely, there exists $\lambda_0$, such that if $\lambda(S)\leq \lambda\leq \lambda_0$, then the surface $S$ is embedded in $\hhh$ and the boundary at infinity is a $K(\lambda)$-quasicircle, with
	$$
	\lim_{\lambda\to 0}K(\lambda)=1\ .
	$$
\end{theo}

For a survey in the asymptotic Plateau problem for minimal surfaces see \textcite{Coskunuzer:2014uk}, for results when the target is negatively curved see \textcite{Lang:2003we}.

\subsection{In variable curvature}  Assume now that the closed hyperbolic manifold $(M,h_0)$ is also equipped with a metric $h$ of curvature less than $-1$.

For any set $\Lambda$ in $\partial_\infty\tilde M$, let $\Env_h(\Lambda)$ be the {\em  convex hull  of $\Lambda$} that is the intersection of all convex subsets  of $\tilde M$  whose closure in $\tilde M\sqcup\partial_\infty\tilde M$  contains $\Lambda$.

Observe again that thanks to Mostow rigidity $K$-quasicircles in the boundary at infinity of the universal cover of $(M,h)$ is a topological notion. We denote by $\tilde M$ the universal cover of $M$.

For the paper being discussed, the authors need to obtain a control between minimal surfaces for both $(M,h)$ and $(M,h_0)$ given in \cite[Theorem 3.1.]{Calegari:2020uo}. This result
follows from results of \textcite{Bowditch:1995wo}.

\begin{theo}\label{theo:var}{\sc [Morse Lemma for minimal surfaces]}
There exists a positive constant~$R$, such that if $S$ and $S_0$ are incompressible minimal surfaces in $M$, for~$h$ and~$h_0$, having the same fundamental group, then
	$$
	d_0(S_0,S)\leq R\ ,
	$$
	where $R$ only depends on $h$, $h_0$ and $\lambda(S_0)$, and $d_0$ is the distance with respect to $h_0$. 
\end{theo}
Correctly extended this result also makes sense for other minimal surfaces than equivariant ones, and could be understood as a Morse Lemma for minimal surfaces.
Since $S$ and $S_0$ are trapped in $\Env_h(\Lambda)$ and $\Env_{h_0}(\Lambda)$ respectively by proposition~\ref{pro:trap}, it is enough to prove 

\begin{prop} For any set $\Lambda$ in $\partial_\infty \tilde M=\partial_\infty \hhh$, we have 
\begin{equation}
d_h(\Env_h(\Lambda),\Env_{h_0}(\Lambda))\leq R\ ,	\label{ineq:CC0}
\end{equation}
for some constant only depending and $h$ and $h_0$. 
\end{prop}

\begin{proof}[Indication of the proof] 
We prove that by introducing the following set: 
for any $p$, let 
$$
\operatorname{Clo}^h_p(\Lambda)\defeq \{\gamma(t)\mid t>0\ , \gamma \hbox{ geodesic for $h$ with }\gamma(0)=p, \gamma(+\infty)\in\Lambda\}\ . 
$$
Then, by a result of \textcite[proposition 2.5.4]{Bowditch:1995wo}, there is some positive constant $R_1(h)$ only depending on $h$, so that for any $p$ in $\Env_h(\Lambda)$,
\begin{equation}
	d_h(\Env_{h}(\Lambda),\operatorname{Clo}^h_p(\Lambda))\leq R_1(h)\ .\label{ineq;CloC}
\end{equation}
Take now geodesics $\gamma$ and $\gamma_0$ joining two points of $\Lambda$, then by the Morse Lemma for geodesics, we can find points $p$ and $p_0$ in $\gamma$ and $\gamma_0$ respectively so that 
\begin{equation}
d_h(p,p_0)\leq R_2\ ,\label{ineq;pp0}
\end{equation}
where $R_2$ only depends on $h$ and $h_0$. As a final ingredient observe that for any $p$ and $q$, 
\begin{equation}
d_h(\operatorname{Clo}^h_p(\Lambda),\operatorname{Clo}^h_q(\Lambda))\leq R_3+d_h(p,q)\ ,\label{ineq;pp3}
\end{equation}
where $R_3$ only depends on $h$. Observing that $p$ belongs to $\Env_h(\Lambda)$ while $p_0$ belongs to $\Env_{h_0}(\Lambda)$, and combining inequalities~\eqref{ineq;CloC}, \eqref{ineq;pp0} and~\eqref{ineq;pp3}, we get the desired inequality~\eqref{ineq:CC0} with
$$
R=R_2+R_2+R_1(h)+R_1(h_0)\ . \qedhere
$$
\end{proof}

\subsubsection{Minimal surfaces and quasi-isometries}

Let again $(M,h_0)$ be a closed hyperbolic $3$-manifold and $h$ another metric on $M$ of curvature less than $-1$.

As a consequence of  proposition~\ref{theo:var}, {\it a priori} bounds on the curvature of minimal surfaces given by proposition~\ref{pro:F-Sch}, Theorem~\ref{theo:Seppi} that gives this result in the hyperbolic case,  and classical arguments about quasi-isometries, we have 
\begin{theo}\label{coro:KK}
There exist positive constants $\epsilon_0$ and $K$ so that the following holds. Assume $h$ is close enough to a hyperbolic metric $h_0$.
Let $S$ be an area minimizing minimal incompressible surface in $(M,h)$, such that the boundary at infinity of $\pi_1(S)$ is $(1+\epsilon_0)$-quasicircle, then 
\begin{enumerate}
	\item the conformal minimal parametrization  $\phi$ from $\hh$ to $S$, is, as a map to the universal cover of $M$, a $K$-quasi-isometric embedding, 
	\item $\phi$ admits an extension to $\mathbf{RP}^1$ which is a $K$-quasi-symmetric map with values in $\partial_\infty \tilde M=\partial_\infty\hhh$.
\end{enumerate} 
\end{theo}
We recall that a map is a {\em $K$-quasi-isometric embedding} if the image of every geodesic is a $K$-quasi-geodesic.

\subsection{The case of fibered manifolds}\label{par:agol}
To conclude our promenade in minimal surfaces in hyperbolic manifolds, and after discussing  almost Fuchsian manifolds, let us say a word about manifolds fibering over the circle, even though none of this will be used further on.

By Agol's Virtual Haken Theorem, any hyperbolic $3$-manifold has a finite cover that fibers over the circle. The fibers of these fibrations are not quasi-Fuchsian, but nevertheless can be represented by minimal surfaces by  theorem~\ref{theo:SY}.

This fibration is {\em taut} by a result of \textcite{Sullivan:1979vg},  which means one  can realize this foliation by minimal surfaces for some metric.

A long standing question was whether this fibration could be realized by a minimal fibration in the hyperbolic metric. The answer to this question is no: there exists $3$-manifolds fibering over the circle, so that the fibers of this fibration cannot be all minimal surfaces. This is a result of \textcite{Hass:2015we} ---see also \textcite{Huang:2019tp}.

\section{Equidistribution in the phase space of minimal surfaces}

\subsection{A phase space for stable minimal surfaces} \label{par:phasespace}
Dealing with solutions of ordinary differential equations, for instance geodesics, we introduced the phase space of the problem, which can be identified  the space of pairs $(x,L)$\footnote{One has to be careful of what we call ``orbit'' to avoid the space to be non Hausdorff}, where $x$ is a point in the orbit $L$ of the ordinary differential equation. One can generalize this construction to solutions of partial differential equations as was done in \textcite{Gromov:1991uv} for minimal surfaces and harmonic mappings and studied in \textcite{Labourie:2005b} for surfaces with constant Gaußian curvature in negatively curved $3$-manifolds. We will do so for stable minimal surfaces and describe measures on this space.

Let $M$ be a closed manifold equipped with a metric~$h$ of curvature less than~$-1$, and~$\tilde M$ its universal cover.

In this section $\hh$ will be the upper half plane model of the hyperbolic plane, which comes with a canonical identification of $\partial_\infty\hh$ with $\mathbf{RP}^1$ and $\iso(\hh)$ with $\psld$. We say a minimal immersion from $\hh$ to $M$ is {\em conformal} if the pullback metric is in the conformal class of the hyperbolic metric.

\begin{defi}	{\sc[Conformal minimal lamination]} 
\phantomsection
Let us fix some small $\epsilon_0$ and  large  constant $K$ so that Theorem~\ref{coro:KK} holds.
\begin{enumerate}
	\item Let  $
\mathcal F_h(\tilde M)
$
be the space of stable minimal conformal immersions of $\hh$ in $\tilde M$ which are $K$-quasi-isometric embeddings, equipped with the topology of uniform convergence on every compact, and 
$\mathcal F_h(M)\defeq\mathcal F_h(\tilde M)/\pi_1(M)$.
	\item   For $\epsilon\leq\epsilon_0$, let $
\mathcal F_h(\tilde M,\epsilon)
$ be the set of those $\phi$ in $\mathcal F_h(\tilde M)$ so that $\phi(\partial_\infty\hh)$ is a $(1+\epsilon)$-quasicircle. Similarly, let $\mathcal F_h(M,\epsilon)\defeq\mathcal F_h(\tilde M,\epsilon)/\pi_1(M)$. 
\end{enumerate}

The space $\mathcal F_h(M)$ together with the action of $\psld$ by precomposition is called the {\em conformal minimal lamination} of $M$. \end{defi}

Finally denote by $\mathcal Q(K)$ the space of $K$-quasicircles in $\partial_\infty\hhh$ equipped with the Gromov--Hausdorff topology. Then
\begin{theo}\phantomsection\label{theo:phasespace}
	\begin{enumerate}
		\item The map from $\mathcal F_h(\tilde M)$ to $\tilde M$, given by $\phi\mapsto \phi(i)$ is a proper map.
		\item The action of $\psld$ by precomposition on $\mathcal F_h(M)$ is continuous and proper.
		\item Moreover, the map $\partial$ from $\mathcal F_h(\tilde M)$ to $\mathcal Q(1+\epsilon)$, which maps $\phi$ to $\phi(\mathbf{RP}^1)$ is continuous and $\psld$-invariant.
			\end{enumerate}
			 
\end{theo}
We may assume that $(1+\epsilon)\leq K$, and we will consider from now on $Q(1+\epsilon)$ as a subset of $Q(K)$ to lighten the notation.

\begin{proof}
	The first point is a rephrasing of proposition~\ref{pro:F-Sch} and~\ref{pro:compac}. The second point and third point also follow from the first and from the fact that any element of $\mathcal F_h(M)$  is a $K$-quasi-isometric embedding.
\end{proof}
 Here is a corollary, using the constants that appear in the previous theorem.
 
 \begin{coro}\label{cor:cont}
 The map $\partial$ gives rise to a continuous map ---also denoted $\partial$--- from $\mathcal F_h(M)/\psld$ to $\mathcal Q(K)$.	
 \end{coro}

A recent preprint of \textcite{Lowe:2020vu} states ---in this language--- that upon small deformation $h$ of the hyperbolic metric the projection from $\mathcal F_h(M,0)$ to $G(M)$ is a homeomorphism. This is a special case of a theorem by \textcite{Gromov:1991uv}.

\subsection{Laminar measures and conformal currents}
This paragraph is an extension of the theory of geodesic currents and invariant measures as in \textcite{Bonahon:1997tl}. 
\begin{defi}
	\begin{enumerate}
		\item A {\em laminar measure} on $\mathcal F_h(\tilde M)/\pi_1(M)$ is a $\psld$-invariant finite measure.
		\item A {\em conformal current} is a $\pi_1(M)$-invariant locally finite measure on $\mathcal Q(K)$.
		\end{enumerate}
\end{defi}

 Here are two examples which are the analogues of the situation for closed
 geodesics. Let $\Gamma$ be a Fuchsian group acting on $\hh$. Let $U$ be a fundamental domain of the action of $\Gamma$ on $\psld$. Let $\rho$ be a representation of $\Gamma$ into $\pi_1(M)$. 
 
 \begin{prop}
 For $\epsilon$ small enough, let $\phi$ be a an element of $\mathcal F_h(\tilde M)$, equivariant under a representation $\rho$ from $\Gamma$ to $\pi_1(M)$, such that $\rho$ is injective and its boundary at infinity is a $(1+\epsilon)$-quasicircle $\Lambda_0$. 
 	 \begin{enumerate}
 	\item  Let $\delta^h_\phi$ be the measure on $\mathcal F_h(M)$  defined by
 	 $$
 	 \int_{\mathcal F_h(M)} f\ {\rm d} \delta^h_\phi=\frac{1}{\Vol(U)}\int_U f(\phi\circ g)\ {\rm d}\mu(g) ,
 	 $$
 	 where $\mu$ is the bi-invariant measure on $\psld$. Then $\delta^h_\phi$ is a $\psld$-invariant probability measure.
 	  \item Let $\delta_\rho$ be  the measure on $\mathcal Q(K)$ defined by  	  $$
 	  \delta_\rho=\sum_{\gamma\in \pi_1(M)/\rho(\Gamma)}\gamma_*\delta_{\Lambda_0}\ ,
 	  $$
 	  where $\delta_{\Lambda_0}$ is the Dirac measure supported on $\Lambda_0$. Then $\delta_\rho$ is a locally finite $\pi_1(M)$-invariant measure on $\mathcal Q(K)$.
 \end{enumerate}
\end{prop}
The only non-trivial point is the fact that $\delta_\rho$ is a locally finite measure. This is checked at the end of the proof of the next proposition.

The next proposition is crucial

\begin{prop}\label{pro:defpi} There exist some positive constant $\epsilon$ and a continuous map  $\pi_h$ from the space of laminar measures \textup{(}up to multiplication by a constant\textup{)} to the space of conformal currents \textup{(}up to multiplication by a positive constant\textup{)} so that 
	$$
	\pi_h(\delta_\phi^h)=\delta_\rho\ ,
	$$ 
	if  $\phi$ is an element of $\mathcal F_h(M,\epsilon)$ equivariant under a representation $\rho$.
	
	Moreover the support of $\pi_h(\mu)$ is the image by $\partial$ of the support of $\mu$.
\end{prop}

\begin{proof} Let us fix a nonnegative function $\Xi$ supported on a bounded neighborhood of  the identity in $\psld$.

	Let $\mu$ be a laminar measure on $\mathcal F_h(M)$. Let us lift $\mu$ to a locally finite $\pi_1(M)$-invariant measure $\tilde\mu$ on $\mathcal F_h(\tilde M)$. Let $\Lambda$ be a $(1+\epsilon)$-quasicircle. Let $\mathcal T$ be the space of triples of pairwise distinct points of $\partial_\infty$. Let $U$ be a small neighborhood of $\Lambda$ in $\mathcal Q(K)$. Let ${\bf F}$ be  a continuous map from $U$ to $\mathcal T$ so that if $(a_0,a_1,a_\infty)={\bf F}(\Lambda)$, then $a_0$, $a_1$, $a_\infty$ belong to~$\Lambda$.
	
	For any $\phi$ in $\mathcal F_h(M)$ so that $\partial\phi=\Lambda$ is in $U$, let $g_\phi$ be the unique element of $\psld$ so that $\phi\circ g_\phi(0,1,\infty)={\bf F}(\Lambda)$. Let then $\xi_{\bf F}$ be the function defined on  $\partial^{-1}U$, by
	$$
	\xi_{\bf F}(\phi)=\Xi(g_\phi)\ .
	$$
	
	Finally let us define, when  $f$ is supported in $U$,
	\begin{equation}
			\int_{\Lambda} f \ {\rm d}\pi_h(\mu) \defeq \int_{\mathcal F_h(\tilde M)} \xi_{\bf F}\cdot (f\circ\partial)\  {\rm d}\tilde\mu .\label{def:pi}
	\end{equation}
	Since $\xi_{\bf F}\cdot (f\circ\partial)$ is compactly supported,  the left hand side is a well-defined finite real number.
	
	Then, one sees that the left hand side  does not depend on the choice of ${\bf F}$ since $\tilde\mu$ is $\psld$-invariant. By construction $\pi(\mu)$ is locally finite since $\tilde\mu$ is,  similarly $\pi(\mu)$ is invariant under the action of $\pi_1(M)$ since $\tilde\mu$ is. 
	Finally the formula shows that $\pi$ is continuous in the weak topology.
	
	We leave the reader check the equality $
	\pi(\delta_\phi^h)=\delta_\rho$. 
\end{proof}

\subsection{Equidistribution}

We can now explain the equidistribution result that follows from the techniques in \textcite{Kahn:2009wh}.

\begin{theo}{\sc [Equidistribution]}\label{theo:KMseq} Let $M$ be a closed hyperbolic $3$-manifold. There exists a sequence $\seq{\delta^{0}}$ of laminar measures on  $\mathcal F_{h_0}(M,1/m)$, such that $\delta^0_m$ is supported on finitely many closed leaves and  so that the sequence of  $\seq{\delta^0}$ converges to $\mu_{Leb}$.
\end{theo}

We will call in the sequel the sequence of measures obtained in this theorem  a {\em Kahn--Markovi{\'c} sequence}.
This result is an extended version of \textcite[Theorem 4.2]{Calegari:2020uo}. 

\begin{proof}[Sketch of the proof] 
	We use a  different geometric presentation than \textcite{Hamenstadt:2015wa} and \textcite{Kahn:2009wh}, developed in \textcite{Kahn:2018wx}. The following convention will hold through this sketch
	\begin{enumerate}
		\item All references in this sketch are from \textcite{Kahn:2018wx}.
		\item $K_i$ will be constants only depending on the closed hyperbolic manifold $M={\hhh}/\pi_1(M)$.
		\item $o(m)$ will denote a function that converges to $0$ when $m$ goes to infinity.
	\item $\alpha^-$ and $\alpha^+$ are the repulsive and attractive fixed points of the element $\alpha$ in $\pi_1(M)$, while $\ell(\alpha)$ is the length of the associated geodesic.

	\item   $\epsilon$ will be a (small) positive constant and $R$ be a (large) positive constant.
	\end{enumerate}  
	
	A {\em tripod} is a triple of pairwise distinct points in $\mathbf{CP}^1\simeq\partial_\infty\hhh$.
	Let $\mathcal T$ be the space of tripods.   The space $\mathcal T$ is canonically identified with the frame bundle $F_{h_0}(\hhh)$ and carries a canonical metric. Every point $x$ in $\mathcal T$ also defines an ideal triangle $\Delta_x$ in $\hhh$ and we denote by $b(x)$ the barycenter of this triangle. We see the {\em barycentric map} $x\mapsto b(x)$ as a projection from $\mathcal T$ to $\hhh$
	
	We remark that there is an open subset $\Delta$ in $\psld$, invariant by the right action of $S^1$, so that 
	\begin{equation*}
		b(\Delta(x))=\Delta_x\ ,
			\end{equation*}
	and thus for any function defined on $\hhh$, we have
	\begin{equation}
\frac{1}{\Vol(\Delta)}\int_\Delta g\circ b\circ u(x)\ {\rm d}\mu(u)=\frac{1}{\pi}\int_{\Delta_x}g\  \area\ ,\label{eq:bD}
	\end{equation}
where ${\rm d}\mu$ is the bi-invariant measure in $\psld$.

A {\em triconnected  pair of tripods} [definition~10.1.1] is a quintuple $(t,s,c_0,c_1,c_2)$ so that $t$ and $s$ are points in $\mathcal T/\pi_1(M)$ and $c_i$ are three homotopy classes of paths from  $t$ to $s$. We denote by $\pi$ the projection from $\mathcal T$ to $\mathcal T/\pi_1(M)$.
	
Let also define $\pi^0$ and $\pi^1$ as the forgetting maps taking values in
the frame bundle $$\pi^0\colon (t,s,c_0,c_1,c_2)\mapsto t\ \ \ ,\ \ \
\pi^1\colon (t,s,c_0,c_1,c_2)\mapsto s\  .$$

	The space of triconnected pair of tripods tripods carries a measure $\mu_{\epsilon,R}$ [definition~12.2.3] satisfying the following property, property which is established by a suitable closing lemma [Theorem~10.3.1][Theorem~9.2.2]:
	 
	Let $\epsilon$ be small enough, then  $R$ large enough.  If $(t,s,c_0,c_1,c_2)$  is in the support of $\mu_{\epsilon,R}$, there exists three  elements $\alpha$, $\beta$ and $\gamma$ of $\pi_1(M)$ so that, 
\begin{enumerate}
		\item let  $t_0=(\alpha^-,\beta^-,\gamma^-)$  and $s_0=(\alpha^-,\alpha(\gamma^-),\beta^-)$, then $\pi(t_0)$ and $\pi(s_0)$ are $K_1\frac{\epsilon}{R}$ close to $t$ and $s$ respectively,  
		 \item $\alpha$, $\beta$ and $\gamma$ are in the homotopy classes of $c_0\cdot c_1^{-1}$, $c_2\cdot c_0^{-1}$ and  $c_1\cdot c_2^{-1}$ respectively.
		\item  The complex cross-ratio of $(\alpha^-, \alpha(\gamma^-),\beta^-,\gamma^-)$ is $K_1\frac{\epsilon}{R}$ close to $R$.  
		\item The  complex length of $\alpha$, $\beta$, $\gamma$ is $K_2\frac{\epsilon}{R}$ close to $2R$.
		
\end{enumerate}
Observe that $\alpha\gamma\beta=1$. Moreover, gluing the two ideal triangles $T_{0}\defeq \Delta_{t_0}$ and $S_0=\Delta_{s_0}$, then taking the quotient by $\pi_1(M)$ one gets a  pleated pair of pants $P$ in $M$, whose fundamental group is generated by  $(\alpha,\beta,\gamma)$. 

Conversely, given three elements in $\pi_1(M)$, $\alpha$, $\gamma$, and $\beta$ so that $\alpha\gamma\beta=1$, one gets  a unique   triconnected pair of tripods (called {\em exact})  by setting $t=(\alpha^-,\beta^-,\gamma^-)$  and $s=(\alpha^-,\alpha(\gamma^-),\beta^-)$ and $c_0$, $c_1$, $c_2$ the obvious paths.

We define a triple $(\alpha,\beta,\gamma)$ satisfying the last two items $(iii)$ and $(iv)$ as an {\em $(\epsilon,R)$-pair of pants}. 
the sequence of measures $\pi^0_*\mu_{\epsilon,R}$ and $\pi^1_*\mu_{\epsilon,R}$ converges to $\mu_{Leb}$ by a mixing argument as $R$ goes to infinity, when $\epsilon$ is fixed.

More precisely, we can choose a sequence $\seq{R}$ going to $\infty$, so that setting $\mu_m=\mu_{\frac{1}{m},R_m}$, then for any  continuous function $f$ on $F_{h_0}(M)$, then  we have [proposition~10.2.6][equation~93]: 
\begin{equation}
	\lim_{m\to\infty}\int (f\circ \pi^0) \ \  {\rm d}\mu_{m}=\lim_{m\to\infty}\int (f\circ \pi^1) \ \  {\rm d}\mu_{m}=\int f {\rm d}\mu_{Leb}\ .
\end{equation}
According to [proposition~18.0.3], and using mixing again, the proof goes by showing that one can approximate  $\mu_{m}$ by measures $\nu_{m}$ with finite support $\{P^m_1,\ldots, P^m_{N_m}\}$, on exact triconnected pair of tripods and rational weights, with the following property.  
 The  sequence of measures $\pi^0_*\nu_{m}$ and $\pi^1_*\nu_{m}$ converges to $\mu_{Leb}$ by a mixing argument as $R$ goes to infinity. In other words, for any function $f$,
\begin{equation}
\lim_{m\to\infty}\frac{1}{N_m}\sum_{i=1}^{N_m}f (\pi_0(P^m_i)) 
  =\lim_{m\to\infty}\frac{1}{N_m}\sum_{i=1}^{N_m}f (\pi_1(P^m_i))=\int f {\rm d}\mu_{Leb}\ .\label{eq:fppp}
\end{equation}
Without loss of generality, we assume that the weight of each $P^i_m$ ---that may appear with multiplicity--- is $1/{N_m}$: in other words  the $P_i$ are counted with multiplicities to have the same weights.

Let us assume now assume  that $f=g\circ b$, where $b$ is the barycentric map from $F_{h_0}(M)$ to  on $M$ and $g$ is defined on $M$. 
Using the invariance  of $\mu_{Leb}$ under the left action of $\psld$ we get that 
\begin{equation}
\lim_{m\to\infty}
\frac{1}{2\pi N_m}\Bigl(\sum_{i=1}^{N_m}\int_{\Delta_0(P^i_m)}g\  \area + \sum_{i=1}^{N_m}\int_{\Delta_1(P^i_m)}g\  \area\Bigr)=\int f {\rm d}\mu_{Leb}\ , \label{eq:intP}
\end{equation}
where $\Delta_0(P)\defeq\Delta_{\pi_0(P)}$ and $\Delta_1(P)\defeq\Delta_{\pi_1(P)}$ are ideal triangles: Indeed we apply equation~\eqref{eq:fppp}, to $f\circ u$, for all  $u$ in $\Delta$ and apply formula~\eqref{eq:bD}.

Moreover, the measure $\nu_m$ satisfies a matching 
condition ({\it} [lemma~14.2.1]) that we now describe. Let $\{P^m_1,\ldots, P^m_{N_m}\}$ be the support of $\nu_{m}$. 

The matching condition is that for each $m$,  we can find by (see [definition~14.1.1][Theorem~16.3.1]) a family of closed surfaces $S_m=(S_1^m,\ldots, S^m_{M_m})$ in $M$  obtained by gluing the pleated pair of pants $\{P^m_1,\ldots, P^m_{N_m}\}$ in $M$ such that every pair of pants appears exactly once. In particular, we get from equation~\eqref{eq:intP} that
\begin{equation}
\lim_{m\to\infty}\frac{1}{2\pi \chi(S_m)}\int_{S_m} g \area =\int f {\rm d}\mu_{Leb}\ . \label{eq:int S}
\end{equation}
For each $S^i_m$, let $\Sigma^i_m$ be the  associated minimal surface and $\Sigma_m$ the union of all $\Sigma^i_m$. As part of the construction, each $S^i_m$ is $(1+o(m))$-almost Fuchsian and thus the projection~$p_m$ from the pleated surface~$S^i_m$ to the minimal surface~$\Sigma_i$ is $(1+ o(m))$-bi-Lipsichitz and satisfies $d(x,p^i_m(x))\leq o(m)$.

It follows from equation~\eqref{eq:int S} that 

\begin{equation*}
\lim_{m\to\infty}\frac{1}{2\pi \chi(\Sigma_m)}\int_{\Sigma_m} g \area =\int f {\rm d}\mu_{Leb}\ . 
\end{equation*}
In other words, setting $\delta^0_m$ the measure on $\mathcal F_{h_0}(M)$ supported on $\Sigma_m$, we have
\begin{equation*}
\lim_{m\to\infty}p_*\delta^0_m= p_*\mu_{Leb}\ .
\end{equation*}
	We can now conclude using corollary~\ref{coro:Ratner3}.	
\end{proof}

\section{Calegari--Marques--Neves asymptotic counting}\label{sec:last}

The article of \textcite{Calegari:2020uo} deals with the counting of surface subgroups. Let  $\Sigma$ be the set of conjugacy classes of surface subgroups in $\pi_1(M)$ for a manifold $M$. For $\Pi$ an element of $\Sigma$ and $h$ a Riemannian metric on $M$, we define
$$
\MinA_h(\Pi)\defeq \inf\{\Area(S)\mid S \hbox{ is an incompressible surface in $M$ with $\pi_1(S)\in\Pi$}\}\ .
$$
Assume now that $M$ is a $3$-manifold that admits a hyperbolic metric. For a  conjugacy class of surface group  $\Pi$, we define $\epsilon(\Pi)$ as
$$
\epsilon(\Pi)\leq \epsilon\ ,
$$
if $\Pi$ is quasi-Fuchsian and the limit circle is a $(1+\epsilon)$-quasicircle. Let then
\begin{align}
S_h(M,T,\epsilon)&\defeq\{\Pi\in\Sigma\mid \MinA_h(\Pi)\leq T\, ,\ \epsilon(\Pi)\leq \epsilon\}\ , \\
E_\epsilon(M,h)
&\defeq \frac{1}{4\pi}\liminf_{T\to\infty}\frac{\log\left(\sharp S_h(M,T,\epsilon)\right)}{T\log T}\ ,\\
	E(M,h)&\defeq \lim_{\epsilon\to 0}E_\epsilon(M,h)\ .
\end{align}
The main result of \textcite{Calegari:2020uo} is then
\begin{theo}\label{theo:main}Let $(M,h_0)$ be a hyperbolic $3$-manifold 
\begin{enumerate}
\item for any Riemannian metric $h$ on 
	$M$, we have $E(M,h)\leq 2{\rm h}_{vol}(M,h)^2$,
\item assume furthermore that $h$ has curvature less than $-1$ then $E(M,h)\geq 2$, with equality if only if the metric  $h$ is hyperbolic.
\end{enumerate}	
\end{theo}
This result is only for counting conjugacy classes of surface groups. Except for the equality case for which the method does not apply, the theorem also holds for counting commensurability classes. We believe the same result holds for commensurability classes after a suitable adaptation of theorem~\ref{theo:equi}.

Thus we have an exact  analogue to \textcite{Bowen:1972} and \textcite{Margulis:1969ve} for the first assertion as well as a rigidity result analogue to \textcite{Hamenstadt:1990tx} for the second assertion.

The  equality case involves a mix of analytical and dynamical properties.

Let us now explain a proof of their result insisting on the use of the phase space of stable minimal surfaces, which synthesizes some of the proofs in \textcite{Calegari:2020uo}.

\subsection{Counting surfaces: from genus to area}
The next proposition is an easy exercise on counting, inclusion and inequalities and is used several times in the proof.
\begin{prop}\label{pro:denum}
Let $\mathcal S$ be a set of conjugacy classes of surface subgroups. Let $\mathcal S(g)$ be the set of those elements of $\mathcal S$ of genus less than $g$, 
\begin{enumerate}
	\item assume that we have a positive  constant $c$, so that  $$\sharp{\mathcal S(g)}\leq (cg)^{2g},$$  as well as constants  $K_0$ and $K_1$,  so that for any $\Pi$ in $\mathcal S$ of genus $g$, we have
\begin{equation*}
	 g- K_1\leq K_0 \MinA_h(\Pi) \ .\end{equation*}
Then 
 	$$
 	\limsup_{T\to\infty}\frac{\log\left(\sharp\left\{\Pi\in\mathcal S\mid\MinA_h(\Pi)\leq  T \right\}\right)}{T\log(T)}\leq 2K_0\ .
 	$$
	\item Assume that we have a constant $c$, so that  $$\sharp{\mathcal S(g)}\geq (cg)^{2g},$$  as well as constants  $K_0$ and $K_1$ so that for any $\Pi$ in $\mathcal S$ of genus $g$, we have
\begin{equation*}
	 g- K_1\geq K_0\MinA_h(\Pi) \ .
\end{equation*}
Then 
 	$$
 	\limsup_{T\to\infty}\frac{\log\left(\sharp\left\{\Pi\in\mathcal S\mid\MinA_h(\Pi)\leq  T \right\}\right)}{T\log(T)}\geq 2 K_0\ .
 	$$ 	
\end{enumerate}
 	\end{prop}

\subsection{The upper bound in variable curvature} 
Our first proposition is 

\begin{prop}
	Let $h$ be any Riemannian metric on a manifold $M$ admitting a
        hyperbolic metric. Then for all $\epsilon$, we have 
	$
	E(M,h,\epsilon)\leq 2 {\rm h}_{top}(M)^2
	$.
\end{prop}
\begin{proof} Recall that 
$$
{\rm h}_{top}(M)={\rm h}_{vol}(M)=\lim_{R\to\infty}\frac{\log\left(\sharp\{\gamma\in\pi_1(M)\mid d_M(\gamma.x,x)\leq R\}\right)}{R}\ .
$$
where $x$ is a point in the universal cover $\tilde M$ of $M$ and $d_M$ the distance in this universal cover. If $S$ is an incompressible surface in $M$ lifting to a disk in $\tilde M$, we have
$$
\{\gamma\in\pi_1(S)\mid d_S(\gamma.x,x)\leq R\}\subset \{\gamma\in\pi_1(M)\mid d_M(\gamma.x,x)\leq R\}\ ,
$$
and thus
$$
{\rm h}_{vol}(S)\leq {\rm h}_{vol}(M)\ .
$$
Combining with Theorem~\ref{theo:entropy} as in \textcite{Katok:1982uv}, we get
$$
{\rm h}_{vol}(M)^2\Area(S)\geq {\rm h}_{vol}(S)^2\Area(S)\geq 4\pi(g-1)\ .
$$
Proposition~\ref{pro:denum} applied to  $\mathcal S=\mathcal S_h(M,\epsilon)$
---using the upper bound  given by theorem~\ref{theo:KMcount}--- yields the result.\end{proof}

\subsection{The upper bound in constant curvature} \label{par:gmg} When $(M,h_0)$ is hyperbolic, the previous upper bound gives $E(M,h_0)\leq 8$, we however 
have a finer estimate.

\begin{prop}
	For the hyperbolic metric $h_0$, we have $E(M,h_0)\leq 2$.
\end{prop}
\begin{proof}
	For a closed minimal surface in $\mathcal S_{h_0}(M,T,\epsilon)$, recall that by Theorem~\ref{theo:Seppi}
	$$
	\lambda(S)\leq C_0\log(1+\epsilon)\eqdef \eta(\epsilon)\ .
	$$
	Thus by the Gauß--Bonnet formula
	$$
	4\pi(g-1)=\int_S(1+\lambda(S)^2)\area\leq (1+\eta(\epsilon)^2)\Area(S) \ .
	$$ 
	From proposition~\ref{pro:denum} applied to $\mathcal{S}=\mathcal S_{h_0}(M,\epsilon)$, we get that 
	$$
	E_\epsilon(M,h_0)\leq 2(1+\eta(\epsilon)^2)\ ,
	$$
	and the result follows.
\end{proof}

\subsection{The lower bound} 
\begin{prop}
		Let $h$ be any Riemannian metric of curvature less than $-1$ on a $3$-manifold $M$ then $
	E(M,h)\geq 2
	$.
	\end{prop}
	\begin{proof} Since the curvature of a minimal surface $S$ is less than that of the ambient manifold by Gauß equation, in that case we get that all minimal surfaces have curvature less than $-1$. Hence by Gauß--Bonnet
	$$ 4\pi(g-1)\geq \Area(S) \ .
$$
Applying proposition~\ref{pro:denum} yields the inequality. \end{proof}

\subsection{The equality case} 
The equality case is the rigidity result for the asymptotic counting of \textcite{Calegari:2020uo}.
\begin{theo}
	Let $h$ be a metric of curvature less than $-1$. Assume that $E(M,h)= 2$. Then $h$ is hyperbolic.
\end{theo}
We assume in the sequel that $h$ is a metric of curvature less than $-1$ with  $E(M,h)= 2$. We give a proof in the simpler case when $h$ is close to a hyperbolic metric, so that Theorem~\ref{coro:KK} holds.

\subsubsection{Finding surfaces which are more and more hyperbolic}
Let $S$ be a closed surface.
 Let $G_S(g)$ be the set of connected finite covers of $S$ of genus less than  $g$, and $G_S$ the set of all connected finite covers. Then we have

\begin{prop}\label{pro:puchta} There is a constant $c_1$ only depending on $S$, so that 
	for  $g$ large enough
	$$
	\sharp(G_S(g))\geq (c_1g)^{2g}\ .
	$$

\end{prop}
\begin{proof}
	By \textcite{Muller:2002tt}, the number of index $n$ subgroups of the fundamental group $\pi_1(S)$ of a genus $g_0$ orientable surface grows like $2n(n!)^{2g_0-2}(1+o(1))$. On the other hand, the genus $g$ of a surface whose fundamental group has index $n$ in $\pi_1(S_0)$ is $g=n(g_0-1)+1$. It follows that 
	$$
	\sharp(G_S(g))\geq (c_1g)^{2g}\ .
	$$
	where $c_1$ only depends on $g_0$.
\end{proof}
\begin{prop}\label{pro:area-lim} For every positive integer $m$,
	let  $S_m$ be a finite union of stable minimal surfaces $S^1_m,\ldots, S^p_m$  with  $\epsilon(S^i_m)\leq \frac{1}{m}$.  Let $\lambda^i_m$ be positive numbers so that 
	$$
	\sum_{i=1}^{p_m}\lambda^i_m=1\ .
	$$
 Then 
	\begin{equation}
			 \lim_{m\to\infty}\sum_{i=1}^{p_m}\lambda^i_m\frac{\Area(S^i_m)}{4\pi(g^i_m-1)}=1\ .\label{eq:mmhh}
	\end{equation}
	where  $g^i_m$ is the genus of  $S^i_m$. \end{prop}

\begin{proof} Since Gauß--Bonnet formula gives $4\pi(g^i_m-1)\geq \Area(S^i_m)$, we have
\begin{equation*}
			 \limsup_{m\to\infty}\sum_{i=1}^{p_m}\lambda^i_m\frac{\Area(S^i_m)}{4\pi(g^i_m-1)}\leq
                         1\ . 
	\end{equation*}
Assume now that the limit in the equation~\eqref{eq:mmhh} is smaller than $k$, with $k<1$,
then  for  arbitrarily large  $m$, we can find  $i_m$ in $\{1,\ldots,p_m\}$ so that 
\begin{equation}
 k4\pi(g^{i_m}_m-1)\geq \Area(S^{i_m}_m) \ .\label{ineq:ksm}	
\end{equation}

As in  the beginning of paragraph~\ref{par:gmg}, let~$G_m$  be the set of 
of connected finite covers of~$S^{i_m}_m$, and~$G_m(g)$ the set of  connected finite covers of genus~$g$. Obviously inequality~\eqref{ineq:ksm} holds for all surfaces in~$G_m$.

Thus we could apply proposition~\ref{pro:denum} ---using proposition~\ref{pro:puchta}--- to get 
$$E_\epsilon(M,h)\geq \frac{2}{k}> 2\ ,$$ and our contradiction. 
\end{proof}

\subsubsection{From more and more hyperbolic to more and more totally geodesic}

The previous proposition has a consequence for laminar measures

\begin{prop}\label{pro:muinf}
Let $\seq{\mu}$ be a sequence of laminar probability measures  on ${\mathcal F_h(M)}$ converging to a laminar measure $\mu_\infty$. Assume  that each  $\mu_m$ is supported on finitely many closed leaves of $\mathcal F_h(M,\frac{1}{m})$.  Then  $\mu_\infty$ is supported on the set of  totally geodesic maps from $\hh$ to $M$ whose boundaries are circles. \end{prop}
\begin{proof} Let us consider the function $F$ on $\mathcal F_h(M)$ defined by  associating  to  a conformal minimal immersion  $\phi$  of  $\hh$ in $M$, the conformal factor $F(\phi)$ of $\phi$ at $i$. 

Assume that  $\phi_0$ is equivariant under a representation $\rho$ of a Fuchsian group $\Gamma$ and if  $S$~is the image of $\phi_0$ in $M$, we have $$
\frac{\Area(S)}{4\pi (g-1)}=\int_{\mathcal F_h(M)}  F{\rm d}\delta_{\phi_0}\ .
$$
Then, proposition~\ref{pro:area-lim} tells us that 
$$
\lim_{m\to\infty}\Bigl(\int_{\mathcal F_h(M)} F \ {\rm d}\mu_m\Bigr)=1\ ,
\hbox{ 
hence
 }
\int_{\mathcal F_h(M)} F \ {\rm d}\mu_\infty =1\ .
$$
By the Ahlfors--Schwartz--Pick Lemma since the curvature of $S$ is less than $-1$, we have the inequality $F\leq 1$. It follows that $\mu_\infty$ is supported on the set $F=1$. In particular for any $\phi$ in the support of $\mu_\infty$, $\phi$ is an isometric immersion.   Thus the curvature of the image of $S$ is $-1$. Since the curvature of $M$ is less than $-1$, this only happens when $S$ is a totally geodesic hyperbolic disk. 

Finally, $\mu_\infty$ is supported on the intersection for all $m$ of $\mathcal F_h(M,\frac{1}{m})$. Since the map $\partial$ is continuous (corollary~\ref{cor:cont}), this intersection  is $\mathcal F_h(M,0)$.

Thus the support of $\mu_\infty$ is contained in  the set of conformal isometries into totally geodesic disks whose boundaries are circles. \end{proof}

\subsubsection{Conclusion}
In this conclusion, we finally use the restricting hypothesis that $h$ is close to a hyperbolic metric so that Theorem~\ref{coro:KK} holds.

Let $\seqm{\delta^0_m}$ be the Kahn--Markovi\'c sequence of laminar measures obtained in Theorem~\ref{theo:equi} for $\mathcal F_{h_0}(M)$. For every $m$, let us write
$$
\delta^0_m=\sum_{i=1}^m \lambda^i_m\delta_{\phi^i_m}\ , \hbox{ with } \sum_{i=1}^m \lambda^i_m=1\ ,
$$
where the $\phi^i_m$ are stable conformal immersions in $(M,h_0)$ equivariant under a cocompact group. Let 
$$
\delta_m=\sum_{i=1}^m \lambda^i_m\delta_{\psi^i_m}\ .
$$
where $\psi^i_m$ is a stable conformal immersion in $(M,h)$ equivariant under a cocompact group and homotopic to $\phi^i_m$.

Let us extract a subsequence so that $\seq{\delta}$ converges to $\mu_\infty$, while, by the Equidistribution Theorem~\ref{theo:equi}, $\seqm{\delta^0_m}$ converges to $\mu_{Leb}$.

Let us consider the projections $\pi_h$ and $\pi_{h_0}$ as in proposition~\ref{pro:defpi} and observe that 
$$
\pi_h(\delta_m)=\pi_{h_0}(\delta^0_m)\ .
$$
Thus by taking limits and using the continuity of $\pi_h$ and $\pi_{h_0}$, we have 
$$
\pi_h(\mu_\infty)=\pi_{h_0}(\mu_{Leb})\ ,
$$
and in particular $\pi_h(\mu_\infty)$ and $\pi_{h_0}(\mu_{Leb})$ have the same support.
We conclude by making two observations
\begin{enumerate}
	\item  the support of $\pi_{h_0}(\mu_{Leb})$ is the set of all circles,
	\item any circle in the support of $\pi_{h_0}(\mu_\infty)$ bounds a totally geodesic hyperbolic plane by proposition~\ref{pro:muinf}.
\end{enumerate}
Thus every circle at infinity bounds a totally geodesic hyperbolic disk in $M$, hence by proposition~\ref{pro:manyhyp},  $h$ is hyperbolic.

\printbibliography
\end{document}
